\documentclass[10pt]{amsart}
\usepackage[leqno]{amsmath}
\usepackage{amssymb,latexsym,soul,cite,amsthm,color,enumitem,graphicx,mathtools,microtype,accents}
\usepackage[colorlinks=true,urlcolor=mulberry,citecolor=mulberry,linkcolor=mulberry,linktocpage,pdfpagelabels,bookmarksnumbered,bookmarksopen]{hyperref}
\definecolor{mulberry}{rgb}{0.77, 0.29, 0.55}
\usepackage[english]{babel}
\usepackage[left=2.5cm,right=2.5cm,top=2.5cm,bottom=2.5cm]{geometry}

\numberwithin{equation}{section}

\newtheorem{theorem}{Theorem}[section]
\theoremstyle{plain}
\newtheorem{lemma}[theorem]{Lemma}
\theoremstyle{plain}
\newtheorem{proposition}[theorem]{Proposition}
\theoremstyle{plain}

\theoremstyle{definition}
\newtheorem{remark}[theorem]{Remark}
\newtheorem{example}[theorem]{Example}

\newcommand{\N}{{\mathbb N}}

\newcommand{\R}{{\mathbb R}}
\newcommand{\eps}{\varepsilon}
\newcommand{\beq}{\begin{equation}}
\newcommand{\eeq}{\end{equation}}
\renewcommand{\le}{\leqslant}
\renewcommand{\ge}{\geqslant}

\newcommand{\w}{W^{s,p}_0(\Omega)}
\newcommand{\fpl}{(-\Delta)_p^s\,}
\newcommand{\ds}{{\rm d}_\Omega^s}

\makeatletter
\newcommand{\leqnomode}{\tagsleft@true}
\newcommand{\reqnomode}{\tagsleft@false}
\makeatother

\newenvironment{enumroman}{\begin{enumerate}

}{\end{enumerate}}

\title[Fractional $p$-Laplacian with jumping reactions]{Multiple solutions for the fractional $p$-Laplacian\\ with jumping reactions}

\author[S.\ Frassu, A.\ Iannizzotto]{Silvia Frassu, Antonio Iannizzotto}

\address[S.\ Frassu, A.\ Iannizzotto]{Department of Mathematics and Computer Science
\newline\indent
University of Cagliari
\newline\indent
Via Ospedale 72, 09124 Cagliari, Italy}
\email{silvia.frassu@unica.it, antonio.iannizzotto@unica.it}


\begin{document}

\begin{abstract}
We study a nonlinear elliptic equation driven by the degenerate fractional $p$-Laplacian, with Dirichlet type condition and a jumping reaction, i.e., $(p-1)$-linear both at infinity and at zero but with different slopes crossing the principal eigenvalue. Under two different sets of hypotheses, entailing different types of asymmetry, we prove the existence of at least two nontrivial solutions. Our method is based on degree theory for monotone operators and nonlinear fractional spectral theory.
\end{abstract}

\maketitle

\begin{center}
Version of \today\
\end{center}

\section{Introduction}\label{sec1}

\noindent
Nonlinear elliptic equations with {\em jumping} (a.k.a.\ asymmetric, or crossing) reactions represent a classical subject of investigation in nonlinear analysis. Such equations can be written in the following general form:
\[-L_p u = f(x,u) \quad \text{in $\Omega$,}\]
coupled with some boundary conditions. Here $\Omega$ is some domain, $L_p$ is an elliptic operator, which is $(p-1)$-homogeneous for some $p>1$ (linear if $p=2$), and $f:\Omega\times\R\to\R$ is a Carath\'eodory function s.t.\ the quotient
\[t\mapsto\frac{f(x,t)}{|t|^{p-2}t}\]
has different finite limits for $|t|\to\infty$ and/or $t\to 0$. The study of such problems goes back to \cite{AP}, with relevant contributions from \cite{F} (where the term 'jumping' was also introduced), \cite{R1}, and \cite{CT} (where a general abstract formulation of the problem was given). All the cited works deal with the semilinear case. In the quasilinear case, we recall the results of \cite{APS1,HP1,MP} (dealing with the Dirichlet $p$-Laplacian), \cite{APS} (dealing with the Neumann $p$-Laplacian). In the nonlocal framework, we recall \cite{IP} (dealing with the fractional Laplacian).
\vskip2pt
\noindent
Since the reaction is asymptotically $(p-1)$-linear at both $\pm\infty$ and $0$, the study of such problems is naturally related to that of the eigenvalue problem for $L_p$. In general, nontrivial solutions appear as soon as the limits above 'jump' over the principal eigenvalue of $L_p$. Existence results can be proved via either variational methods (critical point theory and Morse theory), or topological methods (degree theory).
\vskip2pt
\noindent
In this paper we study the following fractional order nonlinear equation with Dirichlet condition:
\beq\label{dir}
\begin{cases}
\fpl u = f(x,u)& \text{in $\Omega$} \\
u=0 & \text{in $\Omega^c$.}
\end{cases}
\eeq
Here $\Omega\subset\R^N$ ($N\geq2$) is a bounded domain with $C^{1,1}$ boundary, $p \ge 2$, $s\in(0,1)$ s.t.\ $N>ps$, the leading operator is the fractional $p$-Laplacian, defined for all $u:\R^N\to\R$ smooth enough and all $x\in\R^N$ by
\[\fpl u(x)=2\lim_{\eps\to 0^+}\int_{B_\eps^c(x)}\frac{|u(x)-u(y)|^{p-2}(u(x)-u(y))}{|x-y|^{N+ps}}\,dy,\]
and $f:\Omega\times\R\to\R$ is a Carath\'eodory mapping with $(p-1)$-linear growth both at $0$ and at $\pm \infty$, with different slopes (jumping reaction). The operator $\fpl$ is both a nonlocal and a nonlinear one, which for $p=2$ reduces to the well-known fractional Laplacian. The corresponding eigenvalue problem can be stated as follows:
\beq\label{ev}
\begin{cases}
\fpl u = \lambda|u|^{p-2}u & \text{in $\Omega$} \\
u=0 & \text{in $\Omega^c$.}
\end{cases}
\eeq
The eigenvalue problem \eqref{ev} has been studied, for instance, in \cite{IS,LL}, leading to the existence of a diverging sequence of variational (Lusternik-Schnirelmann) eigenvalues
\[0<\lambda_1<\lambda_2\le\ldots\le\lambda_k\le\ldots,\]
with properties analogous to those of the classical $p$-Laplacian. The nonlinear problem \eqref{dir} (or variants of it) was studied in \cite{DQ,DQ2,FI,FRS,ILPS,IL,IMS1}, where asymptotic comparison between $f(x,\cdot)$ and some eigenvalue of the sequence above is often used as a means to the end of proving existence of nontrivial solutions. Most of the cited works use variational methods. In particular, asymmetric reactions ($(p-1)$-superlinear at $\infty$, $(p-1)$-sublinear at $-\infty$) are considered in \cite{IL}.
\vskip2pt
\noindent
Our approach is topological, based on Browder's topological degree for $(S)_+$-maps, and follows \cite{APS1,APS}. We prove multiplicity results for problem \eqref{dir} with jumping reactions, under two different sets of hypotheses:
\begin{itemize}
\item[$(a)$] if the quotient $f(x,t)/(|t|^{p-2}t)$ is asymptotically bounded below $\lambda_1$ for $|t|\to\infty$, and between $\lambda_1$ and $\lambda_2$ for $t\to 0$, then \eqref{dir} has at least two nontrivial solutions (Theorem \ref{jmp});
\item[$(b)$] if the quotient $f(x,t)/(|t|^{p-2}t)$ is asymptotically bounded below $\lambda_1$ for $t\to\infty$, above $\lambda_1$ for both $t\to -\infty,0^+$, and tends to $0$ for $t\to 0^-$, then \eqref{dir} has at least two nontrivial solutions, one of which positive (Theorem \ref{ajmp}).
\end{itemize}
In both cases we do not assume that the limits exist. The proofs are based on a comparison between the operator driving problem \eqref{dir} and the one arising from convenient weighted eigenvalue problems, which preserves Browder's degree by homotopy invariance. In this comparison we use an index formula for $\fpl$ proved in \cite{FRS}, and monotonicity properties of weighted eigenvalues proved in the forthcoming paper \cite{BI}.
\vskip2pt
\noindent
The paper has the following structure: in Section \ref{sec2} we recall the basic notions of the degree theory for demicontinuous $(S)_+$-maps; in Section \ref{sec3} we recall the functional-analytic framework and some well-known results about fractional $p$-Laplacian problems, including weighted eigenvalue problems; in Section \ref{sec4} we deal with case $(a)$; and in Section \ref{sec5} we deal with case $(b)$.
\vskip4pt
\noindent
{\bf Notation:} Throughout the paper, for any $A\subset\R^N$ we shall set $A^c=\R^N\setminus A$. For any two measurable functions $f,g:\Omega\to\R$, $f\le g$ in $\Omega$ will mean that $f(x)\le g(x)$ for a.e.\ $x\in\Omega$ (and similar expressions). The positive (resp., negative) part of $f$ is denoted $f^+$ (resp., $f^-$). If $X$ is an ordered Banach space, then $X_+$ will denote its non-negative order cone. For all $r\in[1,\infty]$, $\|\cdot\|_r$ denotes the standard norm of $L^r(\Omega)$ (or $L^r(\R^N)$, which will be clear from the context). Every function $u$ defined in $\Omega$ will be identified with its $0$-extension to $\R^N$. Moreover, $C$ will denote a positive constant (whose value may change case by case).

\section{Degree theory for $(S)_+$-maps}\label{sec2}

\noindent
Topological degree theory for $(S)_+$-mappings from a Banach space into its dual was introduced by Browder in \cite{B} and subsequent papers, as an infinite-dimensional extension of Brouwer's degree theory, and then generalized in \cite{HP,APS1} to set-valued mappings. We recall here some basic features of such theory, following the general approach of \cite[Section 4.3]{MMP}.
\vskip2pt
\noindent
Let $(X, \|\cdot\|)$ be a separable reflexive Banach space with dual $(X^*, \|\cdot\|_*)$. We say that $A: X \to X^*$ is a $(S)_+$-{\em map}, if for any sequence $(u_n)$ in $X$, $u_n \rightharpoonup u$ in $X$ and
\[\limsup_{n \to \infty} \langle A(u_n), u_n-u\rangle \leq 0\] 
imply $u_n\to u$ (strongly). By Troyanskij's renorming theorem, we can assume that both $X$ and $X^*$ are locally uniformly convex. So, there is a (single-valued) duality map $\mathcal{F}: X \to X^*$ s.t.\ for all $u \in X$
\[\|\mathcal{F}(u)\|^2 = \|u\|^2= \langle \mathcal{F}(u), u\rangle.\] 
Such $\mathcal{F}$ is a $(S)_+$-homeomorphism between $X$ and $X^*$. Also, we remark that if $A:X \to X^*$ is a demicontinuous (i.e., strong to weak$^*$ continuous) $(S)_+$-map and $B:X \to X^*$ is a completely continuous map, then $A+B$ is a demicontinuous $(S)_+$-map.
\vskip2pt
\noindent
We will now define a degree for a triple $(A,U,u^*)$, where $U \subseteq X$ is a bounded open set, $A: \overline{U} \to X^*$ is a demicontinuous $(S)_+$-map, and $u^* \in X^* \setminus A(\partial U)$. First we introduce a Galerkin type approximation. Since $X$ is separable, there exists an increasing sequence $(X_n)$ of finite-dimensional subspaces of $X$ s.t.\
\[\overline{\bigcup_{n=1}^\infty X_n}=X.\]
For all $n \in \N$ we denote $U_n=U \cap X_n$ and define $A_n:\overline{U}_n \to X_n^*$ ($\overline U_n$ denotes the closure of $U_n$ in $X_n$) by setting for all $u\in\overline U_n$, $v\in X_n$
\[\langle A_n(u),v\rangle = \langle A(u),v\rangle_n\]
($\langle\cdot,\cdot\rangle_n$ denotes the duality between $X^*_n$ and $X_n$). By \cite[Proposition 4.38]{MMP}, the Brouwer degree of $A_n$ eventually stabilizes as $n \to \infty$, i.e., there exists $n_0\in\N$ s.t.\ for all $n\ge n_0$ we have $u^*\notin A_n(\partial U_n)$ and
\[\deg_{B}(A_n, U_n, u^*) = \deg_{B}(A_{n_0}, U_{n_0}, u^*).\]
So, we can define the {\em degree} for the triple $(A,U,u^*)$ as 
\[\deg_{(S)_+} (A, U, u^*) = \deg_{B} (A_{n_0}, U_{n_0}, u^*).\]
The integer-valued map $\deg_{(S_+)}$ inherits the main properties of Brouwer's degree. In particular, it is invariant with respect to a special class of homotopies. We say that $h: [0,1] \times \overline{U} \to X^*$ is a $(S)_+$-{\em homotopy}, if $t_n \to t$ in $[0,1]$, $u_n \rightharpoonup u$ in $X$, and
\[\limsup_{n \to \infty} \langle h(t_n,u_n), u_n-u\rangle \leq 0\] 
imply $u_n \to u$ in $X$ and $h(t_n,u_n) \rightharpoonup h(t,u)$ in $X^*$. For instance, if $A, B: \overline{U} \to X^*$ are demicontinuous $(S)_+$-maps, then
\[h(t,u)=(1-t) A(u) + t B(u)\]
defines a $(S)_+$-homotopy \cite[Proposition 4.41]{MMP}. For the reader's convenience, we summarize the properties of $\deg_{(S)_+}$:

\begin{proposition} \label{deg}
{\rm \cite[Theorem 4.42]{MMP}} Let $U \subset X$ be a bounded open set, $A: \overline{U} \to X^*$ be a demicontinuous $(S)_+$-map, $u^* \notin A(\partial U)$. Then:
\begin{enumroman}
\item \label{d1} (normalization) if $u^* \in \mathcal{F}(U)$, then $\deg_{(S)_+} (\mathcal{F}, U, u^*) = 1$;
\item \label{d2} (domain additivity) if $U= U_1 \cup U_2$, with $U_1, U_2 \subset X$ nonempty open sets s.t.\ $U_1 \cap U_2 = \emptyset$ and $u^* \notin A(\partial U_1 \cup \partial U_2)$, then
\[\deg_{(S)_+} (A, U, u^*)= \deg_{(S)_+} (A, U_1, u^*) + \deg_{(S)_+} (A, U_2, u^*);\]
\item \label{d3} (excision) if $C \subset \overline{U}$ is closed s.t.\ $u^* \notin A(C)$, then 
\[\deg_{(S)_+} (A, U \setminus C, u^*)= \deg_{(S)_+} (A, U, u^*);\]
\item \label{d4} (homotopy invariance) if $h: [0,1] \times \overline{U} \to X^*$ is a $(S)_+$-homotopy s.t.\ $u^* \notin h(t,\partial U)$ for all $t \in [0,1]$, then
\[t \mapsto \deg_{(S)_+} (h(t,\cdot), U, u^*)\]
is constant in $[0,1]$;
\item \label{d5} (solution) if $\deg_{(S)_+} (A, U, u^*) \neq 0$, then there exists $u \in U$ s.t.\ $A(u)=u^*$;
\item \label{d6} (boundary dependence) if $B: \overline{U} \to X^*$ is a demicontinuous $(S)_+$-map s.t.\ $A(u)=B(u)$ for all $u \in \partial U$, then
\[\deg_{(S)_+} (A, U, u^*)= \deg_{(S)_+} (B, U, u^*).\]
\end{enumroman}
\end{proposition}

\noindent
We conclude this section by recalling a result on the degree of a potential operator, originally established by Rabinowitz \cite{R} for the Leray-Schauder degree:

\begin{proposition} \label{degmin}
{\rm \cite[Corollary 4.49]{MMP}} Let $\Phi \in C^1(X)$ be a functional s.t.\ $\Phi': X \to X^*$ is a demicontinuous $(S)_+$-map, $u_0 \in X$ be a local minimizer and an isolated critical point of $\Phi$. Then, there exists $\rho_0>0$ s.t.\ for all $\rho\in (0,\rho_0]$
\[\deg_{(S)_+} (\Phi', B_{\rho}(u_0), 0)=1.\]
\end{proposition}

\section{General Dirichlet problems and weighted eigenvalue problems}\label{sec3}

\noindent
In this section we collect some useful results related to the fractional $p$-Laplacian, which we shall exploit in our analysis of problem \eqref{dir}.
\vskip2pt
\noindent
First we fix a functional-analytical framework, following \cite{DNPV, ILPS}. First, for all measurable $u:\R^N\to\R$ we set
\[[u]_{s,p}^p=\iint_{\R^N\times\R^N}\frac{|u(x)-u(y)|^p}{|x-y|^{N+ps}}\,dxdy.\]
Then we define the following fractional Sobolev spaces:
\[W^{s,p}(\R^N)=\big\{u\in L^p(\R^N):\,[u]_{s,p}<\infty\big\},\]
\[\w=\big\{u\in W^{s,p}(\R^N):\,u(x)=0 \ \text{in $\Omega^c$}\big\},\]
the latter being a uniformly convex, separable Banach space with norm $\|u\|=[u]_{s,p}$ and dual space $W^{-s,p'}(\Omega)$ (with norm $\|\cdot\|_{-s,p'})$. Set $p_s^*= Np / (N-ps)$, then the embedding $\w\hookrightarrow L^q(\Omega)$ is continuous for all $q\in[1,p^*_s]$ and compact for all $q\in[1,p^*_s)$.
\vskip2pt
\noindent
On the reaction of \eqref{dir} we make the following general assumption:
\begin{itemize}[leftmargin=1cm]
\item[${\bf H}_0$] $f:\Omega\times\R\to\R$ is a Carath\'{e}odory function, and there exist $c_0>0$, $q\in(1,p^*_s)$ s.t.\ for a.e.\ $x\in\Omega$ and all $t\in\R$
\[|f(x,t)| \leq c_0 (1+|t|^{q-1}).\]
\end{itemize}
Under such hypothesis, the following definition is well posed. We say that $u\in\w$ is a {\em (weak) solution} of \eqref{dir}, if for all $v\in\w$
\[\iint_{\R^N\times\R^N}\frac{|u(x)-u(y)|^{p-2}(u(x)-u(y))(v(x)-v(y))}{|x-y|^{N+ps}}\,dx\,dy = \int_\Omega f(x,u)v\,dx.\]
We have the following a priori estimate for the solutions:

\begin{proposition}\label{bound}
{\rm\cite[Theorem 3.3]{CMS}} Let ${\bf H}_0$ hold, $u \in \w$ be a solution of \eqref{dir}. Then, $u \in L^{\infty}(\Omega)$ with $\|u\|_{\infty} \le C$, for some $C=C(\|u\|)>0$.
\end{proposition}

\noindent
Regularity theory for nonlinear, nonlocal operators is still developing. A major role in such theory is played by the following weighted H\"older spaces, with weight $\ds(x)= \mathrm{dist}(x, \Omega^c)^s$. Set
\[C_s^0(\overline{\Omega})= \Big\{u \in C^0(\overline{\Omega}): \frac{u}{\ds} \ \text{has a continuous extension to} \ \overline{\Omega} \Big\},
\quad \|u\|_{0,s}= \Big\|\frac{u}{\ds}\Big\|_{\infty},\]
and for all $\alpha \in (0,1)$
\[C_s^{\alpha}(\overline{\Omega})= \Big\{u \in C^0(\overline{\Omega}): \frac{u}{\ds} \ \text{has a $\alpha$-H\"older extension to} \ \overline{\Omega}\Big\},
\ 
\|u\|_{\alpha,s}= \|u\|_{0,s} + \sup_{x \neq y} \frac{|u(x)/\ds(x) - u(y)/\ds(y)|}{|x-y|^{\alpha}}.\]
The embedding $C_s^{\alpha}(\overline{\Omega}) \hookrightarrow C_s^0(\overline{\Omega})$ is compact for all $\alpha \in (0,1)$. By \cite[Lemma 5.1]{ILPS}, the positive cone $C_s^0(\overline{\Omega})_+$ of $C_s^0(\overline{\Omega})$ has a nonempty interior given by
\[{\rm int}(C_s^0(\overline{\Omega})_+)= \Big\{u \in C_s^0(\overline{\Omega}):\, \frac{u}{\ds} > 0 \text{ in } \overline{\Omega}\Big\}.\]
Combining Proposition \ref{bound} and \cite[Theorem 1.1]{IMS5}, we have the following global regularity result for the degenerate case $p\ge 2$:

\begin{proposition}\label{reg}
Let ${\bf H}_0$ hold, $u \in \w$ be a solution of \eqref{dir}. Then, $u \in C_s^{\alpha}(\overline{\Omega})$  for some $\alpha \in (0,s]$.
\end{proposition}

\noindent
We define the operators driving \eqref{dir}. For all $u,v\in\w$ we set
\[\langle A(u), v\rangle = \iint_{\R^N \times \R^N} \frac{|u(x)-u(y)|^{p-2} (u(x)-u(y)) (v(x)-v(y))}{|x-y|^{N+ps}}\,dxdy.\]
It is easily seen that $A: \w \to W^{-s,p'}(\Omega)$ is a demicontinuous $(S)_+$-map (see \cite[Lemma 2.1]{FI}, \cite[Lemma 3.2]{FRS}). We recall from \cite[Lemma 2.1]{IL} the following inequality, which holds for all $u\in\w$:
\beq\label{pm}
\|u^\pm\|^p \le \langle A(u),\pm u^\pm\rangle.
\eeq
We also define the Nemytskij operator
\[\langle N_f(u),v \rangle = \int_{\Omega} f(x,u)v\,dx.\]
By ${\bf H}_0$, $N_f: \w \to W^{-s,p'}(\Omega)$ is a completely continuous map. Thus, $A-N_f$ is a demicontinuous $(S)_+$-map. Clearly, any $u \in \w$ is a (weak) solution iff in $W^{-s,p'}(\Omega)$ we have
\[A(u)-N_f(u) = 0.\]
Sometimes we will deal with problem \eqref{dir} variationally. Let us define an energy functional by setting for all $(x,t)\in\Omega\times\R$
\[F(x,t)=\int_0^t f(x,\tau)\,d\tau,\]
and for all $u\in \w$
\[\Phi(u)= \frac{\|u\|^p}{p} -  \int_{\Omega} F(x,u)\,dx.\]
By ${\bf H}_0$, it is easily seen that $\Phi \in C^1(\w)$ with derivative given for all $u\in\w$ by
\[\Phi'(u)=A(u)-N_f(u),\]
so the solutions of \eqref{dir} coincide with the critical points of $\Phi$. By using Proposition \ref{reg} above, we get the following useful result about equivalence of local minimizers of $\Phi$ in the topologies of $\w$ and $C_s^0(\overline{\Omega})$, respectively:

\begin{proposition}\label{svh}
{\rm \cite[Theorem 1.1]{IMS1}} Let ${\bf H}_0$ hold, $u\in\w$. Then, the following are equivalent:
\begin{enumroman}
\item\label{svh1} there exists $\sigma>0$ s.t.\ $\Phi(u+v)\ge\Phi(u)$ for all $v\in\w\cap C_s^0(\overline\Omega)$, $\|v\|_{0,s}\le\sigma$;
\item\label{svh2} there exists $\rho>0$ s.t.\ $\Phi(u+v)\ge\Phi(u)$ for all $v\in\w$, $\|v\| \le\rho$.
\end{enumroman}
\end{proposition}

\noindent
Finally, we recall a strong maximum principle and Hopf's lemma:

\begin{proposition}\label{max}
{\rm\cite[Theorems 1.2, 1.5]{DQ1}} Let ${\bf H}_0$ hold, and $c_1>0$ be s.t.\ for a.e.\ $x \in \Omega$ and all $t \geq 0$
\[f(x,t)\geq -c_1 t^{p-1}.\]
Then, for all $u \in \w_+ \setminus \{0\}$ solution of \eqref{dir} we have $u \in \mathrm{int}(C_s^0(\overline{\Omega})_+)$.
\end{proposition}

\noindent
The rest of this section is devoted to the following weighted eigenvalue problem with $m \in L^{\infty}(\Omega)_+ \setminus \{0\}$ and $\lambda \in \R$:
\beq\label{evm}
\begin{cases}
\fpl u = \lambda m(x) |u|^{p-2}u & \text{in $\Omega$} \\
u=0 & \text{in $\Omega^c$.}
\end{cases}
\eeq
This reduces to \eqref{ev} for $m\equiv 1$, and it falls into the general problem studied in \cite{HS} and in the forthcoming paper \cite{BI} (with possibly singular weights). Set for all $u,v \in \w$
\[\langle K_m(u), v\rangle = \int_{\Omega} m(x) |u|^{p-2} u v\,dx.\]
By \cite[Lemma 3.2]{FRS}, $K_m:\w \to W^{-s,p'}(\Omega)$ is a completely continuous map. So, $A-\lambda K_m$ is a demicontinuous $(S)_+$-map for all $\lambda \in \R$. According to the general definition above, we say that $u \in \w$ is a (weak) solution of \eqref{evm} if in $W^{-s,p'}(\Omega)$ we have
\[A(u)-\lambda K_m(u)=0.\]
So, $\lambda \in \R$ is an {\em eigenvalue} of $\fpl$, with weight $m$, if there exists $u \in \w\setminus \{0\}$ solution of \eqref{evm}, which is then an {\em eigenfunction} associated to $\lambda$. Arguing as in \cite{IS} and following the general scheme of \cite{PAO}, we set
\[\mathcal{S}_p(m) = \Big\{u\in\w:\,\int_\Omega m(x)|u|^p\,dx=1\Big\}.\]
For all $k \in \N$ we set
\[\mathcal{F}_k=\{S \subseteq\mathcal{S}_p(m): S=-S \text{ closed, } i(S) \geq k\},\]
where $i(\cdot)$ denotes the Fadell-Rabinowitz cohomological index, and 
\beq \label{lk}
\lambda_k(m)= \inf_{S \subset \mathcal{F}_k} \sup_{u \in S} \|u\|^p.
\eeq
By means of \eqref{lk} we define a sequence of variational (Lusternik-Schnirelmann) weighted eigenvalues
\[0<\lambda_1(m) < \lambda_2(m) \leq \cdots \leq \lambda_k(m) \leq \cdots \to \infty.\]
If $m\equiv 1$, then we set $\lambda_k=\lambda_k(1)$ (eigenvalues of \eqref{ev}). The following proposition, proved in \cite[Subsection 3.1]{HS} and \cite[Proposition 3.4]{FRS}, summarizes the properties of the principal weighted eigenvalue $\lambda_1(m)$, including a strong monotonicity property with respect to $m$ (see also \cite{DQ}):

\begin{proposition}\label{ev1}
Let $m\in L^\infty(\Omega)_+\setminus\{0\}$. Then, $\lambda_1(m)>0$ is the smallest eigenvalue of $\fpl$ with weight $m$ and has the following variational characterization:
\[\lambda_1(m) = \inf_{u \in \w \setminus \{0\}} \frac{\|u\|^p}{\int_{\Omega} m(x) |u|^p\,dx},\]
the infimum being attained at a unique eigenfunction $\hat u_{1,m}\in\mathcal{S}_p(m)\cap{\rm int}(C^0_s(\overline\Omega)_+)$. Besides, any eigenfunction associated to an eigenvalue $\lambda>\lambda_1(m)$ is nodal. Finally, if $m'\in L^\infty(\Omega)_+\setminus\{0\}$ is s.t.\ $m\le m'$ in $\Omega$ and $m\not\equiv m'$, then $\lambda_1(m)>\lambda_1(m')$.
\end{proposition}

\noindent
Regarding the second eigenvalue $\lambda_2(m)$, from \cite[Theorem 3.7]{HS} and \cite{BI} we have the following properties, including a weaker monotonicity property (analogous to \cite[Proposition 3]{AT} for the $p$-Laplacian):

\begin{proposition}\label{ev2}
Let $m\in L^\infty(\Omega)_+\setminus\{0\}$. Then, $\lambda_2(m)$ is the smallest eigenvalue of $\fpl$ with weight $m$, greater than $\lambda_1(m)$. Besides, if $m'\in L^\infty(\Omega)_+\setminus\{0\}$ is s.t.\ $m<m'$ in $\Omega$, then $\lambda_2(m)>\lambda_2(m')$.
\end{proposition}

\noindent
For the demicontinuous $(S)_+$-map $A-\lambda K_m$ there holds the following index formula (see \cite[Lemma 2, Theorem 2]{APS} for the Neumann $p$-Laplacian):

\begin{proposition}\label{ind}
{\rm \cite[Theorem 3.5]{FRS}} Let $m \in L^{\infty}(\Omega)_+ \setminus \{0\}$, $r>0$. Then
\begin{enumroman}
\item \label{ind1} $\deg_{(S)_+} (A-\lambda K_m, B_r(0),0) = 1$ for all $\lambda \in (0, \lambda_1(m))$;
\item \label{ind2} $\deg_{(S)_+} (A-\lambda K_m, B_r(0),0) = -1$ for all $\lambda \in (\lambda_1(m), \lambda_2(m))$.
\end{enumroman}
\end{proposition}

\noindent
We also recall the following technical property:

\begin{proposition}\label{coer}
{\rm \cite[Lemma 2.7]{IL}} Let $\theta\in L^\infty(\Omega)$ be s.t.\ $\theta\le\lambda_1$ in $\Omega$, $\theta\not\equiv\lambda_1$. Then, there exists $\sigma>0$ s.t.\ for all $u\in\w$
\[\|u\|^p-\int_\Omega\theta(x)|u|^p\,dx \ge \sigma\|u\|^p.\]
\end{proposition}

\noindent
Finally, we consider problem \eqref{evm} with a bounded perturbation $\beta\in L^\infty(\Omega)_+\setminus\{0\}$:
\beq\label{evp}
\begin{cases}
\fpl u = \lambda m(x) |u|^{p-2} u + \beta(x) & \text{in $\Omega$} \\
u=0 & \text{in $\Omega^c$.}
\end{cases}
\eeq
The following result, which will be useful in our study, is analogous to \cite[Lemma 4.1]{DQ2}, dealing with supersolutions for $m=1$ (see \cite[Proposition 4.1]{GGP} for the $p$-Laplacian case):

\begin{lemma}\label{amp}
Let $m,\beta\in L^\infty(\Omega)_+\setminus\{0\}$, $\lambda\ge\lambda_1(m)$, and $u\in\w$ be a solution of \eqref{evp}. Then, $u^-\not\equiv 0$.
\end{lemma}
\begin{proof}
Since $\beta\not\equiv 0$, we clearly have $u\not\equiv 0$. We argue by contradiction, assuming $u\in\w_+\setminus\{0\}$. By Proposition \ref{max}, then we have $u\in{\rm int}(C^0_s(\overline\Omega)_+)$. For all $n\in\N$ set
\[u_n=u+\frac{1}{n}.\]
Let $\hat u_{1,m}\in{\rm int}(C^0_s(\overline\Omega)_+)$ be as in Proposition \ref{ev1}, and for all $n\in\N$, $x\in\R^N$ set
\[v_n(x)=\frac{\hat u_{1,m}^p(x)}{u_n^{p-1}(x)}.\]
It is easily seen that $v_n\in L^p(\Omega)_+$ and $v_n=0$ in $\Omega^c$. Moreover, by applying Lagrange's theorem and $u_n\ge\frac{1}{n}$ in $\R
^N$, for all $x,y\in\R^N$ we have
\begin{align*}
|v_n(x)-v_n(y)| &= \Big|\frac{\hat u_{1,m}^p(x)-\hat u_{1,m}^p(y)}{u_n^{p-1}(x)}+\hat u_{1,m}^p(y)\frac{u_n^{p-1}(y)-u_n^{p-1}(x)}{u_n^{p-1}(x)u_n^{p-1}(y)}\Big| \\
&\le n^{p-1}|\hat u_{1,m}^p(x)-\hat u_{1,m}^p(y)|+\|\hat u_{1,m}\|_\infty^p\frac{|u_n^{p-1}(y)-u_n^{p-1}(x)|}{u_n^{p-1}(x)u_n^{p-1}(y)} \\
&\le pn^{p-1}(\hat u_{1,m}^{p-1}(x)+\hat u_{1,m}^{p-1}(y))|\hat u_{1,m}(x)-\hat u_{1,m}(y)| \\
&+ (p-1)n^{p-1}\|\hat u_{1,m}\|_\infty^p\Big(\frac{1}{u_n(x)}+\frac{1}{u_n(y)}\Big)|u(x)-u(y)| \\
&\le C_n\big(|\hat u_{1,m}(x)-\hat u_{1,m}(y)|+|u(x)-u(y)|\big),
\end{align*}
for some $C_n>0$ depending on $n$. So,
\[\iint_{\R^N\times\R^N}\frac{|v_n(x)-v_n(y)|^p}{|x-y|^{N+ps}}\,dx\,dy \le C\big(\|\hat u_{1,m}\|^p+\|u\|^p),\]
which implies $v_n\in\w_+$. By the discrete Picone's inequality \cite[Proposition 4.2]{BF}, for all $x,y\in\R^N$ we have
\begin{align*}
|\hat u_{1,m}(x)-\hat u_{1,m}(y)|^p &\ge |u_n(x)-u_n(y)|^{p-2}(u_n(x)-u_n(y))\Big(\frac{\hat u_{1,m}^p(x)}{u_n^{p-1}(x)}-\frac{\hat u_{1,m}^p(y)}{u_n^{p-1}(y)}\Big) \\
&= |u(x)-u(y)|^{p-2}(u(x)-u(v))(v_n(x)-v_n(y)).
\end{align*}
Using the inequality above and testing \eqref{evp} with $v_n\in\w_+$ we have
\begin{align}\label{amp1}
\|\hat u_{1,m}\|^p &\ge \iint_{\R^N\times\R^N}\frac{|u(x)-u(y)|^{p-2}(u(x)-u(v))(v_n(x)-v_n(y))}{|x-y|^{N+ps}}\,dx\,dy \\
\nonumber &= \lambda\int_\Omega m(x)u^{p-1}v_n\,dx+\int_\Omega\beta(x)v_n\,dx.
\end{align}
From \eqref{amp1} we have
\[\int_\Omega \beta(x)v_n\,dx \le \|\hat u_{1,m}\|^p < \infty,\]
hence $(\beta v_n)$ is a bounded sequence in $L^1(\Omega)$, tending to $\frac{\beta\hat u_{1,m}^p}{u^{p-1}}$ in $\Omega$. By Lebesgue's dominated convergence theorem we have $\frac{\beta\hat u_{1,m}^p}{u^{p-1}}\in L^1(\Omega)$. Besides, clearly $mu^{p-1}v_n\to m\hat u_{1,m}^p$ in $L^1(\Omega)$. Passing to the limit in \eqref{amp1} as $n\to\infty$, we have
\begin{align*}
\|\hat u_{1,m}\|^p &\ge \lambda\int_\Omega m(x)\hat u_{1,m}^p\,dx+\int_\Omega\beta(x)\frac{\hat u_{1,m}^p}{u^{p-1}}\,dx \\
&> \lambda_1(m)\int_\Omega m(x)\hat u_{1,m}^p\,dx,
\end{align*}
against Proposition \ref{ev1}. Thus, we conclude that $u^-\not\equiv 0$.
\end{proof}

\begin{remark}
Most results in this section also hold in the singular case $p\in(1,2)$. The assumption $p\ge 2$ is only required to have regularity as in Proposition \ref{reg} and the consequent Proposition \ref{svh} (see \cite{IMS1,IMS5}).
\end{remark}

\section{Two nontrivial solutions for jumping reactions}\label{sec4}

\noindent
In this section we study problem \eqref{dir} under the following hypotheses, which imply a symmetric 'jump' over the principal eigenvalue between $0$ and $\pm\infty$:

\begin{itemize}[leftmargin=1cm]
\item[${\bf H}_1$] $f:\Omega\times\R\to\R$ is a Carath\'{e}odory mapping satisfying
\begin{enumroman}
\item\label{h11} for all $M>0$ there exists $a_M \in L^{\infty}(\Omega)_+$ s.t.\ for a.e.\ $x\in\Omega$ and all $|t| \leq M$
\[|f(x,t)| \leq a_M(x);\]
\item\label{h12} there exist $\theta_1, \theta_2 \in L^{\infty}(\Omega)_+$ s.t.\ $\theta_1 \leq \theta_2 \leq \lambda_1$ in $\Omega$, $\theta_2 \not\equiv \lambda_1$, and uniformly for a.e. $x \in \Omega$
\[\theta_1(x) \leq \liminf_{|t|\to \infty}\frac{f(x,t)}{|t|^{p-2}t} \leq \limsup_{|t|\to \infty}\frac{f(x,t)}{|t|^{p-2}t} \leq \theta_2(x);\] 
\item\label{h13} there exist $\eta_1, \eta_2 \in L^{\infty}(\Omega)$ s.t.\ $\lambda_1 \leq \eta_1 \leq \eta_2<\lambda_2$ in $\Omega$, $\eta_1 \not\equiv \lambda_1$, and uniformly for a.e. $x \in \Omega$
\[\eta_1(x) \leq \liminf_{t \to 0}\frac{f(x,t)}{|t|^{p-2}t} \leq \limsup_{t \to 0}\frac{f(x,t)}{|t|^{p-2}t} \leq \eta_2(x).\] 
\end{enumroman}
\end{itemize}

\noindent
Note that we assume non-resonance both at $0$ and $\pm\infty$, with a relevant difference: non-resonance with $\lambda_1$ is only required on a subset of $\Omega$ with positive measure, while non-resonance with $\lambda_2$ must hold on the whole $\Omega$. Clearly, ${\bf H}_1$ implies ${\bf H}_0$ (with $q=p$). So, all the results of Section \ref{sec3} apply here. Besides, from ${\bf H}_1$ \ref{h13} we immediately see that $f(\cdot,0)=0$ in $\Omega$, hence problem \eqref{dir} admits the trivial solution $u=0$. We aim at proving the existence of nontrivial solutions, so we may assume, without loss of generality, that \eqref{dir} has {\em finitely many} solutions.

\begin{example}
The autonomous mapping $f \in C(\R)$ defined by
\[f(t)=\theta|t|^{p-2}t+(\eta-\theta)|t|^{p-2}t\frac{\ln(1+|t|)}{|t|},\]
with $\theta < \lambda_1 < \eta < \lambda_2$, satisfies ${\bf H}_1$.
\end{example}

\noindent
In the following lemmas we study the behavior of the operator $A-N_f$. We begin with an existence result:

\begin{lemma}\label{exi}
If ${\bf H}_1$ holds, then \eqref{dir} has a solution $u_0 \in C_s^{\alpha}(\overline{\Omega})\setminus \{0\}$. Moreover, there exists $\rho_0>0$ s.t.\ for all $\rho\in(0,\rho_0]$
\[\deg_{(S)_+}(A-N_f, B_{\rho}(u_0),0)=1.\]
\end{lemma}
\begin{proof}
By ${\bf H}_1$ \ref{h12} and Proposition \ref{coer}, there exists $\sigma>0$ s.t.\ for all $u\in\w$
\beq\label{exi1}
\|u\|^p-\int_\Omega\theta_2(x)|u|^p\,dx \ge \sigma\|u\|^p.
\eeq
Fix $\eps\in(0,\sigma\lambda_1)$. By ${\bf H}_1$ \ref{h12} we can find $M>0$ s.t.\ for a.e.\ $x\in\Omega$ and all $|t|>M$
\[\frac{f(x,t)}{|t|^{p-2}t} \le \theta_2(x)+\eps.\]
By ${\bf H}_1$ \ref{h11}, for a.e.\ $x\in\Omega$ and all $|t|\le M$ we have
\[|f(x,t)| \le a_M(x).\]
So, for a.e.\ $x\in\Omega$ and all $t>M$ we get
\begin{align*}
F(x,t) &\le \int_0^M|f(x,\tau)|\,d\tau+\int_M^t(\theta_2(x)+\eps)\tau^{p-1}\,d\tau \\
&\le Ma_M(x)+\frac{\theta_2(x)+\eps}{p}(t^p-M^p) \\
&\le \frac{\theta_2(x)+\eps}{p}t^p+C.
\end{align*}
Similar estimates hold for $t\le M$, so for a.e.\ $x\in\Omega$ and all $t\in\R$ we have
\beq\label{exi2}
F(x,t) \le \frac{\theta_2(x)+\eps}{p}|t|^p+C.
\eeq
Define $\Phi\in C^1(\w)$ as in Section \ref{sec3}. By \eqref{exi1}, \eqref{exi2}, and Proposition \ref{ev1} we have for all $u\in\w$
\begin{align*}
\Phi(u) &\ge \frac{\|u\|^p}{p}-\int_\Omega\Big(\frac{\theta_2(x)+\eps}{p}|u|^p+C\Big)\,dx \\
&\ge \frac{\sigma}{p}\|u\|^p-\frac{\eps}{p}\|u\|_p^p-C \\
&\ge \Big(\sigma-\frac{\eps}{\lambda_1}\Big)\frac{\|u\|^p}{p}-C,
\end{align*}
and the latter tends to $\infty$ as $\|u\|\to\infty$. So, $\Phi$ is coercive in $\w$. Plus, it is sequentially weakly l.s.c. Thus, there exists $u_0\in\w$ s.t.
\beq\label{exi3}
\Phi(u_0) = \inf_{u\in\w}\Phi(u) =: \mu_0.
\eeq
Let $\hat u_1\in{\rm int}(C^0_s(\overline\Omega)_+)$ be as in Proposition \ref{ev1} (with $m\equiv 1$). By ${\bf H}_1$ \ref{h13} we have
\[\int_\Omega\eta_1(x)\hat u_1^p\,dx > \lambda_1.\]
Fix now $\eps>0$ s.t.\
\[\eps < \int_\Omega\eta_1(x)\hat u_1^p\,dx-\lambda_1.\]
By ${\bf H}_1$ \ref{h13} there exists $\delta>0$ s.t.\ for a.e.\ $x\in\Omega$ and all $t\in(0,\delta]$
\[\frac{f(x,t)}{t^{p-1}} \ge \eta_1(x)-\eps,\]
hence
\[F(x,t) \ge \frac{\eta_1(x)-\eps}{p}t^p.\]
For all $\tau>0$ small enough we have $0<\tau\hat u_1\le\delta$ in $\Omega$, so, recalling Proposition \ref{ev1}, we have
\begin{align*}
\Phi(\tau\hat u_1) &\le \frac{\tau^p}{p}\|\hat u_1\|^p-\int_\Omega\frac{\eta_1(x)-\eps}{p}(\tau\hat u_1)^p\,dx \\
&= \frac{\tau^p}{p}\Big(\lambda_1-\int_\Omega\eta_1(x)\hat u_1^p\,dx+\eps\Big) < 0.
\end{align*}
Then we have $\mu_0<0$ in \eqref{exi3}, in particular $u_0\neq 0$. From \eqref{exi3} we have $\Phi'(u_0)=0$ in $W^{-s,p'}(\Omega)$, so $u_0$ solves \eqref{dir}. By Proposition \ref{reg} we have $u_0\in C^\alpha_s(\overline\Omega)\setminus\{0\}$.
\vskip2pt
\noindent
Finally, recalling that $\Phi'=A-N_f$ is a demicontinuous $(S)_+$-map and $u_0$ is a local minimizer of $\Phi$ and an isolated critical point (by the assumption that $\Phi$ has only finitely many such points), by Proposition \ref{degmin} there exists $\rho_0>0$ s.t.\ for all $\rho\in(0,\rho_0]$
\[\deg_{(S)_+}(A-N_f,B_\rho(u_0),0) = 1,\]
which concludes the proof.
\end{proof}

\noindent
The next lemma deals with the asymptotic behavior of $A-N_f$:

\begin{lemma}\label{big}
If ${\bf H}_1$ holds, then there exists $R_0>0$ s.t.\ for all $R\ge R_0$
\[\deg_{(S)_+}(A-N_f,B_R(0),0) = 1.\]
\end{lemma}
\begin{proof}
Fix $m_\infty\in L^\infty(\Omega)_+$ s.t.\ $\theta_1\le m_\infty\le\theta_2$ in $\Omega$, and define $K_{m_\infty}:\w\to W^{-s,p'}(\Omega)$ as in Section \ref{sec3}, hence $A-K_{m_\infty}$ is a demicontinuous $(S)_+$-map. Now set for all $(t,u)\in[0,1]\times\w$
\[h_\infty(t,u) = A(u)-(1-t)N_f(u)-tK_{m_\infty}(u).\]
As seen in Section \ref{sec2}, $h_\infty:[0,1]\times\w\to W^{-s,p'}(\Omega)$ is a $(S)_+$-homotopy. We claim that there exists $R_0>0$ s.t.\
\beq\label{big2}
h_\infty(t,u) \neq 0 \ \text{for all $t\in[0,1]$, $\|u\|\ge R_0$.}
\eeq
Arguing by contradiction, assume that there exist sequences $(t_n)$ in $[0,1]$, $(u_n)$ in $\w$ s.t.\ $\|u_n\|\to\infty$ and for all $n\in\N$ we have $h_\infty(t_n,u_n)=0$ in $W^{-s,p'}(\Omega)$, i.e.,
\[A(u_n) = (1-t_n)N_f(u_n)+t_nK_{m_\infty}(u_n).\]
Passing to a subsequence if necessary, we have $t_n\to t$ and $\|u_n\|>0$. Set for all $n\in\N$
\[v_n = \frac{u_n}{\|u_n\|}.\]
The sequence $(v_n)$ is obviously bounded in $\w$, so passing to a further subsequence we have $v_n\rightharpoonup v$ in $\w$, $v_n\to v$ in $L^p(\Omega)$, and $v_n(x)\to v(x)$ for a.e.\ $x\in\Omega$. Dividing the equality above by $\|u_n\|^{p-1}$ we get for all $n\in\N$
\beq\label{big3}
A(v_n) = (1-t_n)g_n+t_n K_{m_\infty}(v_n),
\eeq
where we have set for all\ $x\in\Omega$
\[g_n(x) = \frac{f(x,u_n(x))}{\|u_n\|^{p-1}}.\]
Reasoning as in Lemma \ref{exi} we see that there exists $C>0$ s.t.\ for a.e.\ $x\in\Omega$ and all $t\in\R$
\beq\label{big4}
|f(x,t)| \le C(1+|t|^{p-1}).
\eeq
We focus on the first term on the right-hand side of \eqref{big3}. By \eqref{big4}, we have for all $n\in\N$
\begin{align*}
\int_\Omega |g_n(x)|^{p'}\,dx &\le C\int_\Omega \frac{(1+|u_n|^{p-1})^{p'}}{\|u_n\|^p}\,dx \\
&\le C\frac{1+\|u_n\|_p^p}{\|u_n\|^p},
\end{align*}
and the latter is bounded by the continuous embedding $\w\hookrightarrow L^p(\Omega)$. So, $(g_n)$ is a bounded sequence in $L^{p'}(\Omega)$. Passing to a subsequence, we have $g_n\rightharpoonup g_\infty$ in $L^{p'}(\Omega)$. We claim that there exists $\hat g_\infty\in L^\infty(\Omega)$ s.t.\ in $\Omega$
\beq\label{big5}
g_\infty = \hat g_\infty|v|^{p-2}v, \ \theta_1\le\hat g_\infty\le\theta_2.
\eeq
Indeed, set
\[\Omega^+=\big\{x\in\Omega:\,v(x)>0\big\}, \ \Omega^-=\big\{x\in\Omega:\,v(x)<0\big\}, \ \Omega^0=\big\{x\in\Omega:\,v(x)=0\big\}.\]
Then fix $\eps>0$ and set for all $n\in\N$
\[\Omega^+_{\eps,n} = \Big\{x\in\Omega:\,u_n(x)>0,\,\theta_1(x)-\eps\le\frac{f(x,u_n(x))}{u_n^{p-1}(x)}\le\theta_2(x)+\eps\Big\}.\]
For a.e.\ $x\in\Omega^+$ we have $v_n(x)\to v(x)>0$ as $n\to\infty$, hence $u_n(x)\to\infty$. Recalling ${\bf H}_1$ \ref{h12}, for all $n\in\N$ big enough we have $u_n(x)>0$ and
\[\theta_1(x)-\eps\le\frac{f(x,u_n(x))}{u_n^{p-1}(x)}\le\theta_2(x)+\eps,\]
i.e.,\ $x\in\Omega^+_{\eps,n}$. In other words, $\chi_{\Omega^+_{\eps,n}}\to 1$ a.e.\ in $\Omega^+$, with bounded convergence, which implies $\chi_{\Omega^+_{\eps,n}}g_n\rightharpoonup g_\infty$ in $L^{p'}(\Omega^+)$. By definition of $v_n$, for all $n\in\N$ big enough we have in $\Omega^+$
\[\chi_{\Omega^+_{\eps,n}}(\theta_1-\eps)v_n^{p-1} \le \chi_{\Omega^+_{\eps,n}}g_n \le \chi_{\Omega^+_{\eps,n}}(\theta_2+\eps)v_n^{p-1}.\]
Passing to the limit as $n\to\infty$ we get in $\Omega^+$
\[(\theta_1-\eps)v^{p-1} \le g_\infty \le (\theta_2+\eps)v^{p-1}.\]
Further, letting $\eps\to 0^+$, we get in $\Omega^+$
\[\theta_1v^{p-1} \le g_\infty \le \theta_2v^{p-1},\]
which proves the claim in $\Omega^+$. Similarly, considering the set
\[\Omega^-_{\eps,n} = \Big\{x\in\Omega:\,u_n(x)<0,\,\theta_1(x)-\eps\le\frac{f(x,u_n(x))}{|u_n(x)|^{p-2}u_n(x)}\le\theta_2(x)+\eps\Big\},\]
we get in $\Omega^-$
\[\theta_2|v|^{p-2}v \le g_\infty \le \theta_1|v|^{p-2}v.\]
Finally, for a.e.\ $x\in\Omega^0$ we have $v_n(x)\to 0$, which with $\|u_n\|\to\infty$ and \eqref{big4} implies $g_n(x)\to 0$. So we have $g_\infty=0$ in $\Omega^0$, which completes the argument for \eqref{big5}.
\vskip2pt
\noindent
Now go back to \eqref{big3}, which we test with $v_n-v\in\w$ getting
\[\langle A(v_n),v_n-v\rangle = (1-t_n)\int_\Omega g_n(x)(v_n-v)\,dx+t_n\int_\Omega m_\infty(x)|v_n|^{p-2}v_n(v_n-v)\,dx,\]
and the latter tends to $0$ as $n\to\infty$, so we have
\[\limsup_n\langle A(v_n),v_n-v\rangle \le 0.\]
By the $(S)_+$-property of $A$, we deduce that $v_n\to v$ in $\w$, in particular $\|v\|=1$. Besides, passing to the limit in \eqref{big3} as $n\to\infty$ and applying \eqref{big5}, we have in $W^{-s,p'}(\Omega)$
\[A(v) = \tilde g_\infty|v|^{p-2}v,\]
where we have set for all $x\in\Omega$
\[\tilde g_\infty(x) = (1-t)\hat g_\infty(x)+tm_\infty(x).\]
In other words, $v$ solves the weighted eigenvalue problem
\beq\label{big6}
\begin{cases}
\fpl v = \tilde g_\infty(x)|v|^{p-2}v & \text{in $\Omega$} \\
v = 0 & \text{in $\Omega^c$.}
\end{cases}
\eeq
Clearly $\tilde g_\infty\in L^\infty(\Omega)$, and due to \eqref{big5} and the choice of $m_\infty$ it satisfies $\theta_1\le \tilde g_\infty\le \theta_2$ in $\Omega$. By ${\bf H}_1$ \ref{h12}, then, we have $\tilde g_\infty\le\lambda_1$ in $\Omega$ and $\tilde g_\infty\not\equiv\lambda_1$, hence by Proposition \ref{ev1}
\[\lambda_1(\tilde g_\infty) > \lambda_1(\lambda_1) = 1.\]
Thus, by \eqref{big6}, $v\neq 0$ is an eigenfunction with weight $\tilde g_\infty$, associated to the eigenvalue $1$, against Proposition \ref{ev1}. This proves \eqref{big2}.
\vskip2pt
\noindent
Now we can apply Proposition \ref{deg} \ref{d4} (homotopy invariance), which gives for all $R\ge R_0$
\beq\label{big7}
\deg_{(S)_+}(A-N_f,B_R(0),0) = \deg_{(S)_+}(A-K_{m_\infty},B_R(0),0).
\eeq
To conclude, we compute the degree of $A-K_{m_\infty}$. By ${\bf H}_1$ \ref{h12} we have $m_\infty\le\lambda_1$ in $\Omega$ and $m_\infty\not\equiv\lambda_1$, hence by Proposition \ref{ev1} we have
\[\lambda_1(m_\infty) > \lambda_1(\lambda_1) = 1.\]
Therefore, by Proposition \ref{ind} \ref{ind1} we have for all $R>0$
\[\deg_{(S)_+}(A-K_{m_\infty},B_R(0),0) = 1,\]
which along with \eqref{big7} gives for all $R\ge R_0$
\[\deg_{(S)_+}(A-N_f,B_R(0),0) = 1,\]
thus concluding the proof.
\end{proof}

\noindent
The last lemma deals with the behavior of $A-N_f$ near $0$:

\begin{lemma}\label{sml}
If ${\bf H}_1$ holds, then there exists $r_0>0$ s.t.\ for all $r\in(0,r_0]$
\[\deg_{(S)_+}(A-N_f,B_r(0),0) = -1.\]
\end{lemma}
\begin{proof}
Fix $m_0\in L^\infty(\Omega)$ s.t.\ $\eta_1\le m_0\le \eta_2$ in $\Omega$, and define $K_{m_0}:\w\to W^{-s,p'}(\Omega)$ as in Section \ref{sec3}, hence $A-K_{m_0}$ is a demicontinuous $(S)_+$-map. As in Lemma \ref{big}, we define a $(S)_+$-homotopy $h_0:[0,1]\times\w\to W^{-s,p'}(\Omega)$ by setting for all $(t,u)\in[0,1]\times\w$
\[h_0(t,u) = A(u)-(1-t)N_f(u)-tK_{m_0}(u).\]
We claim that there exists $r_0>0$ s.t.\
\beq\label{sml2}
h_0(t,u) \neq 0 \ \text{for all $t\in[0,1]$, $0<\|u\|\le r_0$.}
\eeq
Arguing as above by contradiction, assume that there exist sequences $(t_n)$ in $[0,1]$, $(u_n)$ in $\w\setminus\{0\}$ s.t.\ $u_n\to 0$ in $\w$ and for all $n\in\N$ we have $h_0(t_n,u_n)=0$ in $W^{-s,p'}(\Omega)$. Set for all $n\in\N$
\[v_n = \frac{u_n}{\|u_n\|}.\]
Passing to a subsequence, we have $t_n\to t$, as well as $v_n\rightharpoonup v$ in $\w$, $v_n\to v$ in $L^p(\Omega)$, and $v_n(x)\to v(x)$ for a.e.\ $x\in\Omega$. Plus, for all $n\in\N$ we have
\beq\label{sml3}
A(v_n) = (1-t_n)g_n+t_n K_{m_0}(v_n),
\eeq
where we have set for all $x\in\Omega$
\[g_n(x) = \frac{f(x,u_n(x))}{\|u_n\|^{p-1}}.\]
By ${\bf H}_1$ \ref{h12} \ref{h13}, we can find $C>0$, $\delta\in(0,1)$ s.t.\ for a.e.\ $x\in\Omega$ and all $t\in\R$ with either $|t|<\delta$ or $|t|>\delta^{-1}$
\[|f(x,t)| \le C|t|^{p-1}.\]
Besides, by ${\bf H}_1$ \ref{h11} with $M=\delta^{-1}>0$, for a.e.\ $x\in\Omega$ and all $\delta\le |t|\le\delta^{-1}$ we have
\[|f(x,t)| \le a_M(x) \le \frac{\|a_M\|_\infty}{\delta^{p-1}}|t|^{p-1}.\]
All in all, taking $C>0$ even bigger if necessary, for a.e.\ $x\in\Omega$ and all $t\in\R$ we have
\[|f(x,t)| \le C|t|^{p-1}.\]
Now for all $n\in\N$ we have
\begin{align*}
\int_\Omega |g_n(x)|^{p'}\,dx &\le \int_\Omega\Big(\frac{C|u_n|^{p-1}}{\|u_n\|^{p-1}}\Big)^{p'}\,dx \\
&\le C\frac{\|u_n\|_p^p}{\|u_n\|^p} = C\|v_n\|_p^p,
\end{align*}
and the latter is bounded by the continuous embedding $\w\hookrightarrow L^p(\Omega)$. So we see that $(g_n)$ is a bounded sequence in $L^{p'}(\Omega)$, hence up to a further subsequence $g_n\rightharpoonup g_0$ in $L^{p'}(\Omega)$. Arguing as in Lemma \ref{big} and defining this time the sets
\[\Omega^+_{\eps,n} = \Big\{x\in\Omega:\,u_n(x)>0,\,\eta_1(x)-\eps\le\frac{f(x,u_n(x))}{u_n^{p-1}(x)}\le \eta_2(x)+\eps\Big\},\]
\[\Omega^-_{\eps,n} = \Big\{x\in\Omega:\,u_n(x)<0,\,\eta_1(x)-\eps\le\frac{f(x,u_n(x))}{|u_n(x)|^{p-2}u_n(x)}\le \eta_2(x)+\eps\Big\}\]
for all $\eps>0$, $n\in\N$, we find $\hat g_0\in L^\infty(\Omega)$ s.t.\ in $\Omega$
\beq\label{sml4}
g_0 = \hat g_0|v|^{p-2}v, \ \eta_1\le\hat g_0\le\eta_2.
\eeq
Testing \eqref{sml3} with $v_n-v\in\w$ and using the $(S)_+$-property of $A$, we see that $v_n\to v$ in $\w$, hence $\|v\|=1$. Passing to the limit in \eqref{sml3} as $n\to\infty$, we see that $v\neq 0$ solves
\beq\label{sml5}
\begin{cases}
\fpl v = \tilde g_0(x)|v|^{p-2}v & \text{in $\Omega$} \\
v = 0 & \text{in $\Omega^c$,}
\end{cases}
\eeq
where we have set for all $x\in\Omega$
\[\tilde g_0(x)=(1-t)\hat g_0(x)+t m_0(x).\]
Clearly, $\tilde g_0\in L^\infty(\Omega)$, and due to \eqref{sml4} and the choice of $m_0$ it satisfies $\eta_1\le\tilde g_0\le\eta_2$ in $\Omega$. By ${\bf H}_1$ \ref{h13}, then, we have $\lambda_1\le\tilde g_0 <\lambda_2$ in $\Omega$ and $\tilde g_0\not\equiv\lambda_1$, so from Proposition \ref{ev1} we have
\[\lambda_1(\tilde g_0) < \lambda_1(\lambda_1) = 1,\]
while by Proposition \ref{ev2} we have
\[\lambda_2(\tilde g_0) > \lambda_2(\lambda_2) = 1.\]
This is against Proposition \ref{ev2}, as problem \eqref{sml5} has no eigenvalue in the interval $(\lambda_1(\tilde g_0),\lambda_2(\tilde g_0))$. So \eqref{sml2} is proved.
\vskip2pt
\noindent
Now we can apply Proposition \ref{deg} \ref{d4} (homotopy invariance), which gives for all $r\in(0,r_0]$
\beq\label{sml6}
\deg_{(S)_+}(A-N_f,B_r(0),0) = \deg_{(S)_+}(A-K_{m_0},B_r(0),0).
\eeq
To conclude, we compute the degree of $A-K_{m_0}$. By ${\bf H}_1$ \ref{h13} we have $\lambda_1\le m_0<\lambda_2$ in $\Omega$ and $m_0\not\equiv\lambda_1$, so by Propositions \ref{ev1}, \ref{ev2} we have
\[\lambda_1(m_0) < 1 < \lambda_2(m_0).\]
Hence, by Proposition \ref{ind} \ref{ind2} we have for all $r>0$
\[
\deg_{(S)_+}(A-K_{m_0},B_r(0),0) = -1,
\]
which along with \eqref{sml6} gives for all $r\in(0,r_0]$
\[\deg_{(S)_+}(A-N_f,B_r(0),0) = -1,\]
thus concluding the proof.
\end{proof}

\noindent
Using the Lemmas above we can prove our first multiplicity result (an analogous result for the Neumann $p$-Laplacian with set-valued reactions is \cite[Theorem 3]{APS}):

\begin{theorem}\label{jmp}
If ${\bf H}_1$ holds, then problem \eqref{dir} has at least two nontrivial solutions $u_0,u_1\in C^\alpha_s(\overline\Omega)\setminus\{0\}$.
\end{theorem}
\begin{proof}
First, from ${\bf H}_1$ we know that $0$ solves \eqref{dir}. From Lemma \ref{exi} we know that there exists a solution $u_0\in C^\alpha_s(\overline\Omega)\setminus\{0\}$ s.t.\ for all $\rho>0$ small enough
\[\deg_{(S)_+}(A-N_f,B_\rho(u_0),0) = 1.\]
Besides, from Lemma \ref{big} we know that for all $R>0$ big enough
\[\deg_{(S)_+}(A-N_f,B_R(0),0) = 1,\]
and from Lemma \ref{sml} that for all $r>0$ small enough
\[\deg_{(S)_+}(A-N_f,B_r(0),0) = -1.\]
Choosing $\rho,r>0$ even smaller and $R>0$ bigger if necessary, we can ensure
\[\overline B_\rho(u_0)\cup\overline B_r(0) \subset B_R(0), \ \overline B_\rho(u_0)\cap\overline B_r(0) = \emptyset.\]
Besides, by our standing assumption that $A-N_f$ vanishes at finitely many points, we can find $\rho,r>0$ s.t.\ $A(u)-N_f(u)\neq 0$ for all $u\in\partial B_\rho(u_0)\cup\partial B_r(0)$. So, by Proposition \ref{deg} \ref{d2} (domain additivity) we have
\begin{align*}
\deg_{(S)_+}(A-N_f,B_R(0),0) &= \deg_{(S)_+}(A-N_f,B_\rho(u_0),0)+\deg_{(S)_+}(A-N_f,B_r(0),0) \\
&+ \deg_{(S)_+}(A-N_f,B_R(0)\setminus\overline{(B_\rho(u_0)\cup B_r(0))},0),
\end{align*}
which amounts to
\[\deg_{(S)_+}(A-N_f,B_R(0)\setminus\overline{(B_\rho(u_0)\cup B_r(0))},0) = 1.\]
By Proposition \ref{deg} \ref{d5} (solution property), there exists $u_1\in B_R(0)\setminus\overline{(B_\rho(u_0)\cup B_r(0))}$ s.t.\ in $W^{-s,p'}(\Omega)$
\[A(u_1)-N_f(u_1) = 0.\]
By Proposition \ref{reg}, finally, we conclude that $u_1\in C^\alpha_s(\overline\Omega)\setminus\{0,u_0\}$ is a second solution of \eqref{dir}.
\end{proof}

\begin{remark}
The proof of Theorem \ref{jmp} can be performed using the properties of the degree in different ways, for instance exploiting Proposition \ref{deg} \ref{d3} (excision property).
\end{remark}

\section{Two nontrivial solutions for doubly asymmetric reactions}\label{sec5}

\noindent
In this section we study problem \eqref{dir} under different hypotheses, implying an asymmetric 'jump', from above the principal eigenvalue to below between $0$ and $\infty$ and vice versa between $0$ and $-\infty$:

\begin{itemize}[leftmargin=1cm]
\item[${\bf H}_2$] $f:\Omega\times\R\to\R$ is a Carath\'{e}odory mapping satisfying
\begin{enumroman}
\item\label{h21} for all $M>0$ there exists $a_M \in L^{\infty}(\Omega)_+$ s.t.\ for a.e. $x \in \Omega$ and $|t| \leq M$
\[|f(x,t)| \leq a_M(x);\]
\item\label{h22} there exist $\theta_1, \theta_2 \in L^{\infty}(\Omega)$ s.t.\ $\theta_1 \leq \theta_2\leq \lambda_1$ in $\Omega$, $\theta_2 \not\equiv \lambda_1$, and uniformly for a.e. $x \in \Omega$
\[\theta_1(x) \leq \liminf_{t \to \infty}\frac{f(x,t)}{t^{p-1}} \leq \limsup_{t \to \infty}\frac{f(x,t)}{t^{p-1}} \leq \theta_2(x);\] 
\item\label{h23} there exist $\eta_1, \eta_2 \in L^{\infty}(\Omega)$ s.t.\ $\lambda_1 \leq \eta_1 \leq \eta_2$ in $\Omega$, $\eta_1 \not\equiv \lambda_1$, 
and uniformly for a.e. $x \in \Omega$
\[\eta_1(x) \leq \liminf_{t \to 0^+} \frac{f(x,t)}{t^{p-1}} \leq \limsup_{t \to 0^+}\frac{f(x,t)}{t^{p-1}} \leq \eta_2(x);\]
\item\label{h24} there exist $\xi_1, \xi_2 \in L^{\infty}(\Omega)$ s.t.\ $\lambda_1\leq \xi_1 \leq \xi_2$ in $\Omega$, $\xi_1 \not\equiv \lambda_1$, 
and uniformly for a.e. $x \in \Omega$
\[\xi_1(x) \leq \liminf_{t \to -\infty}\frac{f(x,t)}{|t|^{p-2}t} \leq \limsup_{t \to - \infty}\frac{f(x,t)}{|t|^{p-2}t} \leq \xi_2(x);\]
\item\label{h25} uniformly for a.e. $x \in \Omega$
\[\lim_{t \to 0^-}\frac{f(x,t)}{|t|^{p-2}t} =0.\]
\end{enumroman}
\end{itemize}

\noindent
Hypotheses ${\bf H}_2$ conjure a doubly asymmetric behavior of $f(x,\cdot)$, which is bounded below $\lambda_1$ at $\infty$, bounded above $\lambda_1$ both at $0^+$ and at $-\infty$ (without resonance on a positive measure subset of $\Omega$), while we assume that it is $(p-1)$-superlinear at $0^-$. Clearly ${\bf H}_2$ implies ${\bf H}_0$ (with $q=p$), so all the results of Section \ref{sec3} apply. As in Section \ref{sec4}, ${\bf H}_2$ \ref{h23} \ref{h25} imply that \eqref{dir} admits the trivial solution $u=0$. Without loss of generality, we may assume that \eqref{dir} has only {\em finitely many} solutions.

\begin{example}
The autonomous mapping $f \in C(\R)$ defined by
\[f(t)=
\begin{cases}
\eta \, t^{p-1} + \frac{2}{\pi}(\theta - \eta) t^{p-1} \arctan(t) &\text{if $t >0$}\\
\xi \,  \frac{2}{\pi}|t|^{p-2} t \arctan(-t) &\text{if $t \leq 0$,} 
\end{cases}\]
with $\theta < \lambda_1$ and $\xi, \eta > \lambda_1$, satisfies ${\bf H}_2$.
\end{example}

\noindent
Dealing with this case, we need to introduce truncated reactions, along with the corresponding operators and functionals. So we set for all $(x,t)\in\Omega\times\R$
\[f_\pm(x,t) = f(x,\pm t^\pm), \ F_\pm(x,t) = \int_0^t f_\pm(x,\tau)\,d\tau.\]
Further, define the completely continuous maps $N^\pm_f:\w\to W^{-s,p'}(\Omega)$ by setting for all $u,v\in\w$
\[\langle N^\pm_f(u),v\rangle = \int_\Omega f_\pm(x,u)v\,dx,\]
and the functionals $\Phi_\pm\in C^1(\w)$ by setting for all $u\in\w$
\[\Phi_\pm(u) = \frac{\|u\|^p}{p}-\int_\Omega F_\pm(x,u)\,dx,\]
satisfying $\Phi'_\pm=A-N^\pm_f$. Finally, for any $m\in L^\infty(\Omega)_+\setminus\{0\}$ we define two completely continuous maps $K^\pm_m:\w\to W^{-s,p'}(\Omega)$ by setting for all $u,v\in\w$
\[\langle K^\pm_m(u),v\rangle = \pm\int_\Omega m(x)(u^\pm)^{p-1}v\,dx.\]
Our first existence results bears a sign information this time:

\begin{lemma}\label{aexi}
If ${\bf H}_2$ holds, then \eqref{dir} has a solution $u_0 \in C_s^{\alpha}(\overline{\Omega})\cap{\rm int}(C^0_s(\overline\Omega)_+)$. Moreover, there exists $\rho_0>0$ s.t.\ for all $\rho\in(0,\rho_0]$
\[\deg_{(S)_+}(A-N_f, B_{\rho}(u_0),0)=1.\]
\end{lemma}
\begin{proof}
We closely follow the argument of Lemma \ref{exi}. Using ${\bf H}_2$ \ref{h21} \ref{h22} and Proposition \ref{coer} we prove that $\Phi_+$ is coercive in $\w$. Besides, it is sequentially weakly l.s.c. Thus, there exists $u_0\in\w$ s.t.\
\beq\label{aexi1}
\Phi_+(u_0) = \inf_{u\in\w}\Phi_+(u_0) =: \mu^+_0.
\eeq
Then, using ${\bf H}_2$ \ref{h23}, we see that $\mu^+_0<0$, hence $u_0\neq 0$. By \eqref{aexi1}, we have $\Phi'_+(u_0)=0$ in $W^{-s,p'}(\Omega)$, i.e., for all $v\in\w$
\beq\label{aexi2}
\langle A(u_0),v\rangle = \int_\Omega f_+(x,u_0)v\,dx.
\eeq
Testing \eqref{aexi2} with $-u_0^-\in\w$, by \eqref{pm} we have
\begin{align*}
\|u_0^-\|^p &\le \langle A(u_0),-u^-_0\rangle \\
&= \int_\Omega f_+(x,u_0)(-u_0^-)\,dx = 0.
\end{align*}
so $u_0\ge 0$ in $\Omega$. Therefore, \eqref{aexi2} rephrases as \eqref{dir}. Since $u_0\in\w_+\setminus\{0\}$ solves \eqref{dir}, by Propositions \ref{reg}, \ref{max} we have $u_0 \in C_s^{\alpha}(\overline{\Omega})\cap{\rm int}(C^0_s(\overline\Omega)_+)$.
\vskip2pt
\noindent
Since $\Phi=\Phi_+$ in $\w_+$, from \eqref{aexi1} we see that $u_0\in{\rm int}(C^0_s(\overline\Omega)_+)$ is a local minimizer of $\Phi$ in $C^0_s(\overline\Omega)$. So, by Proposition \ref{svh}, it is as well a local minimizer of $\Phi$ in $\w$. By our standing assumption that $\Phi$ has only finitely many critical points, $u_0$ is an isolated critical point of $\Phi$. So, by Proposition \ref{degmin}, there exists $\rho_0>0$ s.t.\ for all $\rho\in(0,\rho_0]$
\[\deg_{(S)_+}(A-N_f,B_\rho(u_0),0) = 1,\]
which concludes the proof.
\end{proof}

\noindent
Again we study the asymptotic behavior of $A-N_f$, which mainly relies on the growth of $f(x,\cdot)$ at $-\infty$:

\begin{lemma}\label{abig}
If ${\bf H}_2$ holds, then there exists $R_0>0$ s.t.\ for all $R\ge R_0$
\[\deg_{(S)_+}(A-N_f,B_R(0),0) = 0.\]
\end{lemma}
\begin{proof}
Fix $m_\infty\in L^\infty(\Omega)_+$ s.t.\ $\xi_1\le m_\infty\le\xi_2$ in $\Omega$, and define $K_{m_\infty}^-:\w\to W^{-s,p'}(\Omega)$ as above. The first part of the proof follows that of Lemma \ref{big}. We define a $(S)_+$-homotopy $h^-_\infty:[0,1]\times\w\to W^{-s,p'}(\Omega)$ by setting for all $(t,u)\in[0,1]\times\w$
\[h^-_\infty(t,u) = A(u)-(1-t)N_f(u)-tK_{m_\infty}^-(u).\]
We claim that there exists $R_0>0$ s.t.\
\beq\label{abig1}
h^-_\infty(t,u) \neq 0 \ \text{for all $t\in[0,1]$, $\|u\|\ge R_0$.}
\eeq
Arguing by contradiction, assume that there exist sequences $(t_n)$ in $[0,1]$, $(u_n)$ in $\w$ s.t.\ $\|u_n\|\to\infty$ and for all $n\in\N$ we have $h^-_\infty(t_n,u_n)=0$ in $W^{-s,p'}(\Omega)$, i.e.,
\beq\label{abig2}
A(u_n) = (1-t_n)N_f(u_n)+t_nK^-_{m_\infty}(u_n).
\eeq
By ${\bf H}_2$ \ref{h22} and Proposition \ref{coer}, there exists $\sigma>0$ s.t.\ for all $u\in\w$
\[\|u\|^p-\int_\Omega\theta_2(x)|u|^p\,dx \ge \sigma\|u\|^p.\]
Fix $\eps\in(0,\sigma\lambda_1)$. Then, by ${\bf H}_2$ \ref{h22} we can find $M>0$ s.t.\ for a.e.\ $x\in\Omega$ and all $t\ge M$
\[f(x,t) \le (\theta_2(x)+\eps)t^{p-1}.\]
Also, for a.e.\ $x\in\Omega$ and all $t\in[0,M]$ we have by ${\bf H}_2$ \ref{h21}
\[f(x,t) \le a_M(x).\]
All in all, we can find $C>0$ s.t.\ for a.e.\ $x\in\Omega$ and all $t\ge 0$
\[f(x,t) \le (\theta_2(x)+\eps)t^{p-1}+C.\]
By applying \eqref{pm} and testing \eqref{abig2} with $u^+_n\in\w_+$, we get for all $n\in\N$
\begin{align*}
\|u^+_n\|^p &\le \langle A(u_n),u^+_n\rangle \\
&= (1-t_n)\int_\Omega f(x,u_n)u^+_n\,dx-t_n\int_\Omega m_\infty(x)(u^-_n)^{p-1}u^+_n\,dx \\
&\le \int_\Omega(\theta_2(x)+\eps)(u^+_n)^p\,dx+C\|u^+_n\|_1,
\end{align*}
which, along with the continuous embedding $\w\hookrightarrow L^1(\Omega)$, implies
\[\Big(\sigma-\frac{\eps}{\lambda_1}\Big)\|u^+_n\|^p \le C\|u^+_n\|.\]
Hence $(u^+_n)$ is bounded in $\w$. Besides, by the triangle inequality we have for all $n\in\N$
\[\|u^-_n\| \ge \|u_n\|-\|u^+_n\|,\]
and the latter tends to $\infty$ as $n\to\infty$. So we have $\|u^-_n\|\to\infty$. Passing if necessary to a subsequence, we have $t_n\to t$ and $\|u_n\|>0$. Set for all $n\in\N$
\[v_n = \frac{u_n}{\|u_n\|}.\]
Since $(v_n)$ is bounded in $\w$, passing to a further subsequence we have $v_n\rightharpoonup v$ in $\w$, $v_n\to v$ in $L^p(\Omega)$, and $v_n(x)\to v(x)$ for a.e.\ $x\in\Omega$. Note that
\[\|v^+_n\| = \frac{\|u^+_n\|}{\|u_n\|} \to 0\]
as $n\to\infty$. So, for a.e.\ $x\in\Omega$ we have
\[v^-_n(x) = v^+_n(x)-v_n(x) \to -v(x),\]
which implies $v(x)\le 0$ in $\Omega$. Dividing \eqref{abig2} by $\|u_n\|^{p-1}$ we have for all $n\in\N$
\beq\label{abig3}
A(v_n) = (1-t_n)g_n+t_nK^-_{m_\infty}(v_n),
\eeq
where we have set for all $x\in\Omega$
\[g_n(x) = \frac{f(x,u_n(x))}{\|u_n\|^{p-1}}.\]
Reasoning as in Lemma \ref{big} we see that $(g_n)$ is bounded in $L^{p'}(\Omega)$, hence, passing to a subsequence, $g_n\rightharpoonup g_\infty$ in $L^{p'}(\Omega)$. We will now prove that there exists $\hat g_\infty\in L^\infty(\Omega)$ s.t.\ in $\Omega$
\beq\label{abig4}
g_\infty = \hat g_\infty|v|^{p-2}v, \ \xi_1\le\hat g_\infty\le\xi_2.
\eeq
Indeed, recall that $v\le 0$ in $\Omega$. Set
\[\Omega^- = \big\{x\in\Omega:\,v(x)<0\big\}, \ \Omega^0 = \big\{x\in\Omega:\,v(x)=0\big\}.\]
Then, for all $\eps>0$, $n\in\N$ set
\[\Omega^-_{\eps,n} = \Big\{x\in\Omega:\,u_n(x)<0,\,\xi_1(x)-\eps\le\frac{f(x,u_n(x))}{|u_n(x)|^{p-2}u_n(x)}\le\xi_2(x)+\eps\Big\}.\]
By ${\bf H}_2$ \ref{h24} we have $\chi_{\Omega^-_{\eps,n}}\to 1$ in $\Omega^-$ with bounded convergence, hence $\chi_{\Omega^-_{\eps,n}}g_n\rightharpoonup g_\infty$ in $L^{p'}(\Omega^-)$. Besides, by definition of $v_n$, for all $\eps>0$ and all $n\in\N$ big enough, in $\Omega^-$ we have $v_n<0$ and
\[\chi_{\Omega^-_{\eps,n}}(\xi_2+\eps)|v_n|^{p-2}v_n \le \chi_{\Omega^-_{\eps,n}}g_n \le \chi_{\Omega^-_{\eps,n}}(\xi_1-\eps)|v_n|^{p-2}v_n.\]
Passing to the limit as $n\to\infty$ and $\eps\to 0^+$, we get in $\Omega^-$
\[\xi_2|v|^{p-2}v \le g_\infty \le \xi_1|v|^{p-2}v.\]
Similarly, we get $g_\infty=0$ in $\Omega^0$, which completes the argument for \eqref{abig4}.
\vskip2pt
\noindent
Now we test \eqref{abig3} with $v_n-v\in\w$, so we get for all $n\in\N$
\[\langle A(v_n),v_n-v\rangle = (1-t_n)\int_\Omega g_n(x)(v_n-v)\,dx-t_n\int_\Omega m_\infty(x)(v_n^-)^{p-1}(v_n-v)\,dx.\]
The latter tends to $0$ as $n\to\infty$, so by the $(S)_+$-property of $A$ we have $v_n\to v$ in $\w$, hence in particular $\|v\|=1$. Passing to the limit in \eqref{abig3} as $n\to\infty$ and using \eqref{abig4}, we see that $v$ solves the weighted eigenvalue problem
\beq\label{abig5}
\begin{cases}
\fpl v = \tilde g_\infty(x)|v|^{p-2}v & \text{in $\Omega$} \\
v = 0 & \text{in $\Omega^c$,}
\end{cases}
\eeq
where we have set for all $x\in\Omega$
\[\tilde g_\infty(x) = (1-t)\hat g_\infty(x)+tm_\infty(x).\]
Clearly $\tilde g_\infty\in L^\infty(\Omega)$, in addition by \eqref{abig4} and the choice of $m_\infty$ we have $\xi_1\le\tilde g_\infty\le\xi_2$ in $\Omega$, hence by ${\bf H}_2$ \ref{h24} $\tilde g_\infty\ge\lambda_1$ in $\Omega$ with $\tilde g_\infty\not\equiv\lambda_1$. By Proposition \ref{ev1} we have
\[\lambda_1(\tilde g_\infty) < \lambda_1(\lambda_1) = 1.\]
So, $v$ is a non-principal eigenfunction of \eqref{abig5}, hence nodal (again by Proposition \ref{ev1}), against $v\le 0$ in $\Omega$. The contradiction proves \eqref{abig1}.
\vskip2pt
\noindent
We can now apply Proposition \ref{deg} \ref{d4} (homotopy invariance) and get for all $R\ge R_0$
\beq\label{abig6}
\deg_{(S)_+}(A-N_f,B_R(0),0) = \deg_{(S)_+}(A-K^-_{m_\infty},B_R(0),0).
\eeq
So we are led to computing the degree of $A-K^-_{m_\infty}$, which this time cannot be done directly. Fix $\beta_\infty\in L^\infty(\Omega)_+\setminus\{0\}$, and for all $(t,u)\in[0,1]\times\w$ set
\[\hat h_\infty(t,u) = A(u)-K^-_{m_\infty}(u)+t\beta_\infty\]
(here we identify $\beta_\infty$ with an element of $W^{-s,p'}(\Omega)$). Clearly $\hat h_\infty:[0,1]\times\w\to W^{-s,p'}(\Omega)$ is a $(S)_+$-homotopy. We claim that for all $t\in[0,1]$ and all $u\in\w\setminus\{0\}$
\beq\label{abig7}
\hat h_\infty(t,u)\neq 0.
\eeq
Arguing by contradiction, let $t\in[0,1]$, $u\in\w\setminus\{0\}$ be s.t.\ in $W^{-s,p'}(\Omega)$
\beq\label{abig8}
A(u) = K^-_{m_\infty}(u)-t\beta_\infty.
\eeq
We distinguish two cases:
\begin{itemize}[leftmargin=1cm]
\item[$(a)$] If $t=0$, then \eqref{abig8} rephrases as
\[A(u) = K^-_{m_\infty}(u).\]
Testing with $u^+\in\w$ and applying \eqref{pm}, we have
\begin{align*}
\|u^+\|^p &\le \langle A(u),u^+\rangle \\
&= -\int_\Omega m_\infty(x)(u^-)^{p-1}u^+\,dx = 0,
\end{align*}
so $u\le 0$ in $\Omega$. Then $u$ solves in fact
\[\begin{cases}
\fpl u = m_\infty(x)|u|^{p-2}u & \text{in $\Omega$} \\
u = 0 & \text{in $\Omega^c$.}
\end{cases}\]
By ${\bf H}_2$ \ref{h24} and the choice of $m_\infty$ we have $m_\infty\ge\lambda_1$ in $\Omega$ and $m_\infty\not\equiv\lambda_1$, hence
\[\lambda_1(m_\infty) < \lambda_1(\lambda_1) = 1.\]
So, $u\in -\w_+\setminus\{0\}$ is a non-principal eigenfunction with weight $m_\infty$, hence nodal by Proposition \ref{ev1}, a contradiction. 
\item[$(b)$] If $t\in(0,1]$, then testing \eqref{abig8} with $u^+\in\w_+$ and applying \eqref{pm} we have
\begin{align*}
\|u^+\|^p &\le \langle A(u),u^+\rangle \\
&= -\int_\Omega m_\infty(x)(u^-)^{p-1}u^+\,dx-t\int_\Omega\beta(x)u^+\,dx \le 0,
\end{align*}
so again $u\le 0$ in $\Omega$. Then, $-u\in\w_+$ satisfies
\[\begin{cases}
\fpl(-u) = m_\infty(x)(-u)^{p-1}+t\beta_\infty(x) & \text{in $\Omega$} \\
-u = 0 & \text{in $\Omega^c$.}
\end{cases}\]
As above $\lambda_1(m_\infty)<1$, so this violates Lemma \ref{amp}.
\end{itemize}
In both cases we reach a contradiction, thus proving \eqref{abig7}. Again by Proposition \ref{deg} \ref{d4} (homotopy invariance) we have for all $R>0$
\beq\label{abig9}
\deg_{(S)_+}(A-K^-_{m_\infty},B_R(0),0) = \deg_{(S)_+}(A-K^-_{m_\infty}+\beta_\infty,B_R(0),0).
\eeq
Finally, for all $R>0$ we have
\beq\label{abig10}
\deg_{(S)_+}(A-K^-_{m_\infty}+\beta_\infty,B_R(0),0) = 0.
\eeq
Arguing by contradiction, assume that the degree above does not vanish for some $R>0$. Then, by Proposition \ref{deg} \ref{d5} (solution property) there exists $u\in B_R(0)$ s.t.\ in $W^{-s,p'}(\Omega)$
\[A(u) = K^-_{m_\infty}(u)-\beta_\infty,\]
namely,
\[\begin{cases}
\fpl u = -m_\infty(x)(u^-)^{p-1}-\beta_\infty(x) & \text{in $\Omega$} \\
u = 0 & \text{in $\Omega^c$.}
\end{cases}\]
Arguing as in case $(b)$ above we see that $u\le 0$ in $\Omega$ and reach a contradiction to Lemma \ref{amp}, thus proving \eqref{abig10}.
\vskip2pt
\noindent
Now, concatenating \eqref{abig6}, \eqref{abig9}, and \eqref{abig10}, we have for all $R\ge R_0$
\[\deg_{(S)_+}(A-N_f,B_R(0),0) = 0,\]
thus concluding the proof.
\end{proof}

\noindent
We complete the picture by studying the behavior of $A-N_f$ near $0$, which is affected by the behavior of $f(x,\cdot)$ near $0^+$:

\begin{lemma}\label{asml}
If ${\bf H}_2$ holds, then there exists $r_0>0$ s.t.\ for all $r\in(0,r_0]$
\[\deg_{(S)_+}(A-N_f,B_r(0),0) = 0.\]
\end{lemma}
\begin{proof}
Fix $m_0\in L^\infty(\Omega)$ s.t.\ $\eta_1\le m_0\le\eta_2$ in $\Omega$, and define the map $K^+_{m_0}:\w\to W^{-s,p'}(\Omega)$ as above. At first we follow the proof of Lemma \ref{sml}. We define a $(S)_+$-homotopy $h^+_0:[0,1]\times\w\to W^{-s,p'}(\Omega)$ by setting for all $(t,u)\in[0,1]\times\w$
\[h^+_0(t,u) = A(u)-(1-t)N_f(u)-tK^+_{m_0}(u).\]
We claim that there exists $r_0>0$ s.t.\ 
\beq\label{asml1}
h^+_0(t,u) \neq 0 \ \text{for all $t\in[0,1]$, $0<\|u\|\le r_0$.}
\eeq
Arguing by contradiction, assume that there exist sequences $(t_n)$ in $[0,1]$, $(u_n)$ in $\w\setminus\{0\}$ s.t.\ $u_n\to 0$ in $\w$ and $h^+_0(t_n,u_n)=0$ in $W^{-s,p'}(\Omega)$ for all $n\in\N$. Set for all $n\in\N$
\[v_n = \frac{u_n}{\|u_n\|}.\]
Passing if necessary to a subsequence, we have $t_n\to t$ as well as $v_n\rightharpoonup v$ in $\w$, $v_n\to v$ in $L^p(\Omega)$, and $v_n(x)\to v(x)$ for a.e.\ $x\in\Omega$. Besides, for all $n\in\N$ we have in $W^{-s,p'}(\Omega)$
\beq\label{asml2}
A(v_n) = (1-t_n)g_n+t_nK^+_{m_0}(v_n),
\eeq
where we have set for all $x\in\Omega$
\[g_n(x) = \frac{f(x,u_n(x))}{\|u_n\|^{p-1}}.\]
Reasoning as in Lemma \ref{sml} we see that $(g_n)$ is a bounded sequence in $L^{p'}(\Omega)$, so passing to a further subsequence we have $g_n\rightharpoonup g_0$ in $L^{p'}(\Omega)$. We claim that there exists $\hat g_0\in L^\infty(\Omega)$ s.t.\ in $\Omega$
\beq\label{asml3}
g_0 = \hat g_0(v^+)^{p-1}, \ \eta_1\le\hat g_0\le\eta_2.
\eeq
Indeed, define the set
\[\Omega^+ = \big\{x\in\Omega:\,v(x)>0\big\},\]
and for all $\eps>0$, $n\in\N$
\[\Omega^+_{\eps,n} = \Big\{x\in\Omega:\,u_n(x)>0,\,\eta_1(x)-\eps\le\frac{f(x,u_n(x))}{u_n^{p-1}(x)}\le\eta_2(x)+\eps\Big\}.\]
By ${\bf H}_2$ \ref{h23} we have $\chi_{\Omega^+_{\eps,n}}\to 1$ a.e.\ in $\Omega^+$, with bounded convergence, hence $\chi_{\Omega^+_{\eps,n}}g_n\rightharpoonup g_0$ in $L^{p'}(\Omega^+)$. Recalling the definition of $v_n$, for all $\eps>0$ and all $n\in\N$ big enough, in $\Omega^+$ we have $v_n>0$ and
\[\chi_{\Omega^+_{\eps,n}}(\eta_1-\eps)v_n^{p-1} \le \chi_{\Omega^+_{\eps,n}}g_n \le \chi_{\Omega^+_{\eps,n}}(\eta_2+\eps)v_n^{p-1}.\]
Passing to the limit as $n\to\infty$ and then as $\eps\to 0^+$, we get in $\Omega^+$
\[\eta_1 v^{p-1} \le g_0 \le \eta_2 v^{p-1}.\]
Similarly, using ${\bf H}_2$ \ref{h25} we get $g_0=0$ in $\Omega\setminus\Omega^+$, which completes the argument for \eqref{asml3}.
\vskip2pt
\noindent
Now we go back to \eqref{asml2}. Testing with $v_n-v\in\w$, we have
\[\langle A(v_n),v_n-v\rangle = (1-t_n)\int_\Omega g_n(x)(v_n-v)\,dx+t_n\int_\Omega m_0(x)(v_n^+)^{p-1}(v_n-v)\,dx,\]
and the latter tends to $0$ as $n\to\infty$. So, by the $(S)_+$-property of $A$ we have $v_n\to v$ in $\w$, hence $\|v\|=1$. Passing to the limit in \eqref{asml2} as $n\to\infty$ and using \eqref{asml3}, we get in $W^{-s,p'}(\Omega)$
\beq\label{asml4}
A(v) = \tilde g_0(v^+)^{p-1},
\eeq
where we have set for all $x\in\Omega$
\[\tilde g_0(x) = (1-t)\hat g_0+tm_0(x).\]
Clearly $\tilde g_0\in L^\infty(\Omega)$, and by \eqref{asml3} and the choice of $m_0$ we have $\eta_1\le\tilde g_0\le\eta_2$ in $\Omega$. Testing \eqref{asml4} with $-v^-\in\w$ and applying \eqref{pm}, we have
\begin{align*}
\|v^-\|^p &\le \langle A(v),-v^-\rangle \\
&= \int_\Omega\tilde g_0(x)(v^+)^{p-1}(-v^-)\,dx = 0,
\end{align*}
hence $v\ge 0$ in $\Omega$. We can therefore rephrase \eqref{asml4} as the weighted eigenvalue problem
\beq\label{asml5}
\begin{cases}
\fpl v = \tilde g_0(x)v^{p-1} & \text{in $\Omega$} \\
v = 0 & \text{in $\Omega^c$.}
\end{cases}
\eeq
By ${\bf H}_2$ \ref{h23} we have $\tilde g_0\ge\lambda_1$ in $\Omega$ with $\tilde g_0\not\equiv\lambda_1$, so by Proposition \ref{ev1}
\[\lambda_1(\tilde g_0) < \lambda_1(\lambda_1) = 1.\]
Thus, $v\neq 0$ is a non-principal eigenfunction of \eqref{asml5}, hence nodal (again by Proposition \ref{ev1}), against $v\ge 0$. This contradiction proves \eqref{asml1}.
\vskip2pt
\noindent
We apply Proposition \ref{deg} \ref{d4} (homotopy invariance) and get for all $r\in(0,r_0]$
\beq\label{asml6}
\deg_{(S)_+}(A-N_f,B_r(0),0) = \deg_{(S)_+}(A-K^+_{m_0},B_r(0),0).
\eeq
Now there remains to compute the right-hand side. We proceed as in Lemma \ref{abig}, fixing $\beta_0\in L^\infty(\Omega)_+\setminus\{0\}$ and setting for all $(t,u)\in[0,1]\times\w$
\[\hat h^+_0(t,u) = A(u)-K^+_{m_0}(u)-t\beta_0.\]
Clearly, $\hat h^+_0:[0,1]\times\w\to W^{-s,p'}(\Omega)$ is a $(S)_+$-homotopy. Again we claim that for all $t\in[0,1]$ and all $u\in\w\setminus\{0\}$
\beq\label{asml7}
\hat h^+_0(t,u) \neq 0.
\eeq
Arguing by contradiction, let $t\in[0,1]$, $u\in\w\setminus\{0\}$ be s.t.\ in $W^{-s,p'}(\Omega)$
\beq\label{asml8}
A(u) = K^+_{m_0}(u)+t\beta_0.
\eeq
We distinguish two cases:
\begin{itemize}[leftmargin=1cm]
\item[$(a)$] If $t=0$, then \eqref{asml8} rephrases as
\[A(u) = K^+_{m_0}(u).\]
Testing with $-u^-\in\w$ and applying \eqref{pm} we have
\begin{align*}
\|u^-\|^p &\le \langle A(u),-u^-\rangle \\
&= \int_\Omega m_0(x)(u^+)^{p-1}(-u^-)\,dx = 0,
\end{align*}
so $u\ge 0$ in $\Omega$. Then $u$ solves in fact
\[\begin{cases}
\fpl u = m_0(x)u^{p-1} & \text{in $\Omega$} \\
u = 0 & \text{in $\Omega^c$.}
\end{cases}\]
By ${\bf H}_2$ \ref{h23} we have $m_0\ge\lambda_1$ in $\Omega$ with $m_0\not\equiv\lambda_1$, so by Proposition \ref{ev1}
\[\lambda_1(m_0) < \lambda_1(\lambda_1) = 1.\]
So, $u$ is a non-principal eigenfunction with weight $m_0$, hence nodal, against $u\ge 0$.
\item[$(b)$] If $t\in(0,1]$, then testing \eqref{asml8} with $-u^-\in\w$ and applying \eqref{pm} we have
\begin{align*}
\|u^-\|^p &\le \langle A(u),-u^-\rangle \\
&= \int_\Omega m_0(x)(u^+)^{p-1}(-u^-)\,dx+t\int_\Omega\beta_0(x)(-u^-)\,dx \le 0,
\end{align*}
so $u\ge 0$ in $\Omega$. Then \eqref{asml8} becomes
\[\begin{cases}
\fpl u = m_0(x)u^{p-1}+t\beta_0(x) & \text{in $\Omega$} \\
u = 0 & \text{in $\Omega^c$,}
\end{cases}\]
with $1>\lambda_1(m_0)$ as above and $u\ge 0$ in $\Omega$, against Lemma \ref{amp}.
\end{itemize}
In both cases we reach a contradiction, thus proving \eqref{asml7}. This in turn allows us to apply Proposition \ref{deg} \ref{d4} (homotopy invariance) and have for all $r>0$
\beq\label{asml9}
\deg_{(S)_+}(A-K^+_{m_0},B_r(0),0) = \deg_{(S)_+}(A-K^+_{m_0}-\beta_0,B_r(0),0).
\eeq
To conclude, we claim that for all $r>0$
\beq\label{asml10}
\deg_{(S)_+}(A-K^+_{m_0}-\beta_0,B_r(0),0) = 0.
\eeq
Arguing by contradiction, assume that for some $r>0$
\[\deg_{(S)_+}(A-K^+_{m_0}-\beta_0,B_r(0),0) \neq 0.\]
By Proposition \ref{deg} \ref{d5} (solution property) there exists $u\in B_r(0)$ s.t.\
\[\begin{cases}
\fpl u = m_0(x)(u^+)^{p-1}+\beta_0(x) & \text{in $\Omega$} \\
u = 0 & \text{in $\Omega^c$.}
\end{cases}\]
Arguing as in case $(b)$ we see that $u\ge 0$ in $\Omega$, violating Lemma \ref{amp}. So \eqref{asml10} is proved.
\vskip2pt
\noindent
Finally, concatenating \eqref{asml6}, \eqref{asml9}, and \eqref{asml10} we get for all $r\in(0,r_0]$
\[\deg_{(S)_+}(A-N_f,B_r(0),0) = 0,\]
which concludes the proof.
\end{proof}

\noindent
Our multiplicity result for this case is the following:

\begin{theorem}\label{ajmp}
If ${\bf H}_2$ holds, then problem \eqref{dir} has at least two nontrivial solutions $u_0 \in C_s^{\alpha}(\overline{\Omega})\cap{\rm int}(C^0_s(\overline\Omega)_+)$, $u_1\in C_s^{\alpha}(\overline{\Omega})\setminus\{0\}$.
\end{theorem}
\begin{proof}
The proof is similar to that of Theorem \ref{jmp}. First, from ${\bf H}_2$ we know that $0$ solves \eqref{dir}. By Lemma \ref{aexi} there exists a solution $u_0 \in C_s^{\alpha}(\overline{\Omega})\cap{\rm int}(C^0_s(\overline\Omega)_+)$ s.t.\ for all $\rho>0$ small enough
\[\deg_{(S)_+}(A-N_f,B_\rho(u_0),0) = 1.\]
Besides, by Lemma \ref{abig} we have for all $R>0$ big enough
\[\deg_{(S)_+}(A-N_f,B_R(0),0) = 0,\]
and by Lemma \ref{asml} we have for all $r>0$ small enough
\[\deg_{(S)_+}(A-N_f,B_r(0),0) = 0.\]
Choosing $\rho,r>0$ even smaller and $R>0$ bigger if necessary, we can ensure
\[\overline B_\rho(u_0)\cup\overline B_r(0) \subset B_R(0), \ \overline B_\rho(u_0)\cap\overline B_r(0) = \emptyset.\]
By our standing assumption that $A-N_f$ vanishes at finitely many points, we can find $\rho,r>0$ s.t.\ $A(u)-N_f(u)\neq 0$ for all $u\in\partial B_\rho(u_0)\cup\partial B_r(0)$. So, by Proposition \ref{deg} \ref{d2} (domain additivity) we have
\begin{align*}
\deg_{(S)_+}(A-N_f,B_R(0),0) &= \deg_{(S)_+}(A-N_f,B_\rho(u_0),0)+\deg_{(S)_+}(A-N_f,B_r(0),0) \\
&+ \deg_{(S)_+}(A-N_f,B_R(0)\setminus\overline{(B_\rho(u_0)\cup B_r(0))},0),
\end{align*}
which amounts to
\[\deg_{(S)_+}(A-N_f,B_R(0)\setminus\overline{(B_\rho(u_0)\cup B_r(0))},0) = -1.\]
By Proposition \ref{deg} \ref{d5} (solution property), there exists $u_1\in B_R(0)\setminus\overline{(B_\rho(u_0)\cup B_r(0))}$ s.t.\ in $W^{-s,p'}(\Omega)$
\[A(u_1)-N_f(u_1) = 0.\]
By Proposition \ref{reg}, finally, we conclude that $u_1\in C^\alpha_s(\overline\Omega)\setminus\{0,u_0\}$ is a second solution of \eqref{dir}.
\end{proof}

\begin{remark}
Formally, Theorem \ref{ajmp} above is analogous to \cite[Theorem 32]{APS1}. Nevertheless, the nonlocal nature of the operator $\fpl$ bears significant differences. The main issue is that, in general, for $u\in\w$ we have
\[\fpl u \neq \fpl u^+-\fpl u^-,\]
which demands a different technique in dealing with the positive and negative parts, with respect to the case of the $p$-Laplacian (compare for instance \cite[Proposition 30]{APS1} to our Lemma \ref{abig}).
\end{remark}

\vskip4pt
\noindent
{\bf Acknowledgement.} Both authors are members of GNAMPA (Gruppo Nazionale per l'Analisi Matematica, la Probabilit\`a e le loro Applicazioni) of INdAM (Istituto Nazionale di Alta Matematica 'Francesco Severi') and are supported by the research project \emph{Evolutive and Stationary Partial Differential Equations with a Focus on Biomathematics}, funded by Fondazione di Sardegna (2019). A.\ Iannizzotto is also supported by the grant PRIN-2017AYM8XW: \emph{Nonlinear Differential Problems via Variational, Topological and Set-valued Methods}. We warmly thank L.\ Brasco for precious suggestions about the eigenvalue problem.

\end{document}